\documentclass[10pt]{article}

\usepackage{amsmath,amsthm,amssymb}
\usepackage{mathtools}
\usepackage[margin=35mm]{geometry}
\usepackage[noadjust]{cite}

\usepackage{xcolor}
\usepackage{hyperref}
\definecolor{darkgreen}{rgb}{0,0.4,0}
\definecolor{BrickRed}{rgb}{0.65,0.08,0}
\hypersetup{colorlinks=true,linkcolor=blue,citecolor=red,filecolor=BrickRed,urlcolor=darkgreen}
\newcommand{\tL}{\mathtt 1}                            % typewriter 1
\newcommand{\tO}{\mathtt 0}                            % typewriter 0

\newcommand{\N}{\mathbb{N}}
\newcommand{\Z}{\mathbb{Z}}
\newcommand{\dens}{\operatorname{dens}}

\DeclareMathOperator{\e}{e}     % exp(ix)

\newcommand{\digitsum}{\mathsf s}
\newcommand{\digitsumdifference}{d_{\mathsf s}}
\newcommand{\digitsumdensity}{\sigma_{\mathsf s}}
\newcommand{\eqdef}{\coloneqq}

\newcommand{\blocknumber}{\mathsf r}

\newcommand{\alphap}{\alpha^*}
\newcommand{\betap}{\beta^*}
\newcommand{\alphaq}{\widetilde\alpha}
\newcommand{\betaq}{\widetilde\beta}

\newcommand{\Sp}{S^*}
\newcommand{\Sq}{\widetilde S}

\newtheorem{intro}{Theorem}
\newenvironment{introtheorem}[1]
  {\renewcommand\theintro{#1}\intro}
  {\endintro}

\newtheorem{theorem}{Theorem}[section]

\newtheorem{proposition}[theorem]{Proposition}

\newtheorem{lemma}[theorem]{Lemma}

\theoremstyle{remark}
\newtheorem*{remark}{Remark}
\newtheorem*{remarks}{Remarks}

\numberwithin{equation}{section}

\definecolor {colorLS}{rgb}{0.9,0.2,0.9}
\definecolor {colorBS}{rgb}{0.2,0.6,0.2}
\definecolor {colorMagenta}{rgb}{0.6,0,0.3}

\title{Block occurrences in the binary expansion}

\author{
\begin{tabular}{c@{\hspace{0.5cm}}c@{\hspace{0.5cm}}c}
\begin{tabular}{c}Bartosz Sobolewski\\Jagiellonian University, Krak\'ow\\
Poland
\end{tabular}
&
\begin{tabular}{c}Lukas Spiegelhofer
\\Montanuniversit\"at Leoben\\
Austria
\end{tabular}
\end{tabular}
}
\date{}
\begin{document}
\maketitle
\begin{abstract}
The binary sum-of-digits function $\digitsum$ returns the number of ones in the binary expansion of a nonnegative integer.
Cusick's Hamming weight conjecture states that, for all integers $t\geq 0$, the set of nonnegative integers $n$ such that $\digitsum(n+t)\geq \digitsum(n)$ has asymptotic density strictly larger than $1/2$.

We are concerned with the block-additive function $\blocknumber$ returning the number of (overlapping) occurrences of the block $\tL \tL$ in the binary expansion of $n$.
The main result of this paper is a central limit-type theorem for the difference $\blocknumber(n+t)-\blocknumber(n)$:
the corresponding probability function is uniformly close to a Gaussian, where the uniform error tends to $0$ as the number of blocks of ones in the binary expansion of $t$ tends to $\infty$.

%\LS{I am OK with the notation $r(n)$}
%$c_t$ of integers $n\geq 0$ such that 
%\[s(n+t)\geq s(n).\]
%In 2011, he conjectured that $c_t>1/2$ for all $t$ --- the binary sum of digits should, more often than not, weakly increase when a constant is added.
%In this paper, we prove that there exists an explicit constant $M_0$ such that indeed $c_t>1/2$ if the binary expansion of $t$ contains at least $M_0$ maximal blocks of contiguous ones, leaving open only the ``initial cases'' --- few maximal blocks of ones --- of this conjecture.
%Moreover, we sharpen a result by Emme and Hubert (2019), proving that the difference $s(n+t)-s(n)$ behaves according to a Gaussian distribution, up to an error tending to $0$ as the number of maximal blocks of ones in the binary expansion of $t$ grows.
\end{abstract}

\renewcommand{\thefootnote}{\fnsymbol{footnote}} 
\footnotetext{\emph{2020 Mathematics Subject Classification.} Primary: 11A63, 05A20; Secondary: 05A16,11T71}
% 	11A63   	Radix representation; digital problems
%   05A20     Combinatorial inequalities

%   05A16     Asymptotic enumeration
%   11T71     Algebraic coding theory; cryptography
\footnotetext{\emph{Key words and phrases.} Cusick conjecture, Hamming weight, sum of digits}
\footnotetext{
Bartosz Sobolewski was supported by the grant of the National Science Centre (NCN), Poland, no.\ UMO-2020/37/N/ST1/02655.\\
Lukas Spiegelhofer acknowledges support by the FWF--ANR project ArithRand (grant numbers I4945-N and ANR-20-CE91-0006), and by the FWF project P36137-N.
}
\renewcommand{\thefootnote}{\arabic{footnote}}

\section{Introduction}
Every nonnegative integer $n$ admits a unique representation
\begin{equation}\label{eqn_binary_expansion}
n=\sum_{j\ge0}\delta_j2^n,
\end{equation}
where $\delta_j\in\{0,1\}$,
which is called the \emph{binary expansion} of $n$.
Each digit $\delta_j$ is therefore a function $\delta_j(\cdot)$ of $n$. %, and we write $\delta_j(n)\eqdef\delta_j$.
The central question we ask is the following:
\begin{equation}\label{eqn_guiding_question}
\mbox{How does the binary expansion behave under addition?}
\end{equation}
As a first step towards a possible answer to the question, we consider
the \emph{binary sum-of-digits function} $\digitsum$ of a nonnegative integer $n$, defined by
\[\digitsum(n)\eqdef\sum_{j\ge0}\delta_j(n),\]
and the differences
\[\digitsumdifference(t,n)\eqdef \digitsum(n+t)-\digitsum(n).\]

The sum-of-digits function $\digitsum$ appears when the $2$-valuation of binomial coefficients is considered.
We have the identity
\begin{equation}\label{eqn_legendre}
\digitsum(n+t)-\digitsum(n)=\digitsum(t)-\nu_2\biggl(\binom{n+t}t\biggr),
\end{equation}
where
$\nu_2(a)=\max\{k\in\mathbb N:2^k\mid a\}$,
which follows from Legendre's formula.
The $2$-valuation of $\binom{n+t}t$ is also the number of \emph{carries} that appears when adding $n$ and $t$ in binary (Kummer~\cite{Kummer1852}).
It appears that both sides in~\eqref{eqn_legendre}
are nonnegative more than half of the time ---
more precisely, T.~W.~Cusick's Hamming weight Conjecture~\cite{DrmotaKauersSpiegelhofer2016} states that for each integer $t\ge0$, we have
%$\digitsum(n+t)<\digitsum(n)$ in strictly less than one %half of all cases.
%In other words,
\begin{equation}\label{eqn_cusick}
\sum_{j\ge0}\digitsumdensity(t,j)>1/2,
\end{equation}
where
\begin{equation}\label{eqn_delta_def}
\digitsumdensity(t,j)\eqdef\dens\bigl\{n \in \N: \digitsumdifference(t,n)=j\bigr\},
\end{equation}
and $\dens A$ is the asymptotic density of a set $A\subseteq \N$.
The asymptotic density exists in this case, as the sets in~\eqref{eqn_delta_def} are unions of arithmetic progressions (B\'esineau~\cite{Besineau1972}),
and
\[\sum_{j\in\mathbb Z}\digitsumdensity(t,j)=1\]
for all $t\in\mathbb N$.
We have the recurrence~\cite{DrmotaKauersSpiegelhofer2016}
%{{{eqn_digitsum_density_recurrence
\begin{equation}\label{eqn_digitsum_density_recurrence}
\begin{aligned}
\digitsumdensity(1,j)&=\begin{cases}0,&j>1,\\2^{j-2},&j\leq 1,\end{cases}\\
\digitsumdensity(2t,j)&=\digitsumdensity(t,j),\\
\digitsumdensity(2t+1,j)&=\frac 12 \digitsumdensity(t,j-1)+\frac 12\digitsumdensity(t+1,j+1),
\end{aligned}
\end{equation}
%}}}eqn:density_recurrence
valid for all integers $t\ge1$ and $j$.
Making essential use of this recurrence,
the second author and Wallner~\cite{SpiegelhoferWallner2021}
proved an \emph{almost-solution} to Cusick's conjecture.

\begin{introtheorem}{A}
Under the hypothesis that $\tO\tL$ occurs at least $N_0$ times in the binary expansion of $n$, where $N_0$ can be made explicit,
the statement~\eqref{eqn_cusick} holds.
\end{introtheorem}

T.~W.~Cusick remarked upon learning about this result (private communication) that the ``hard cases'' of his conjecture remain open!

In the same paper~\cite[Theorem~1.2]{SpiegelhoferWallner2021} a central limit-type result is proved.
%\LS{I don't know how to use ``Theorem B'' with the \texttt{\textbackslash ref} command.}

\begin{introtheorem}{B}\label{thm_digitsum_normal}
For integers $t\geq 1$, let us define
\begin{align*}
\kappa(1)=2,\qquad \kappa(2t)=\kappa(t),\qquad
\kappa(2t+1)=\frac{\kappa(t)+\kappa(t+1)}2+1.
\end{align*}
Assume that $\tO\tL$ appears $N$ times in the binary expansion of the positive integer $t$, and $N$ is larger than some constant $N_0$.
Then the estimate
\begin{equation*}%\label{eqn_delta_normal}
\digitsumdensity(t,j)
=\frac 1{\sqrt{2\pi\kappa(t)}}\exp\left(-\frac{j^2}{2\kappa(t)}\right)
+\mathcal O\bigl(N^{-1}(\log N)^4\bigr)
\end{equation*}
holds for all integers $j$.
The multiplicative constant of the error term can be made explicit.
\end{introtheorem}
This theorem sharpens the main result in~\cite{EmmeHubert2018}, see also~\cite{EmmePrikhodko2017}.

The value $\kappa(t)$ is the variance of the probability distribution given by the densities $\digitsumdensity(t,j)$ (where $j\in\mathbb Z$).
It equals the second moment, as the mean is zero:
\begin{equation}\label{eqn_digitsum_mean_zero}
\sum_{j\in\mathbb Z}j\digitsumdensity(t,j)=0.
\end{equation}
Note that the function $\tfrac12\kappa$ appears in another context too: it is the \emph{discrepancy of the van der Corput sequence}~\cite{DrmotaLarcherPillichshammer2005}.

Returning to Cusick's conjecture~\eqref{eqn_cusick},
we note that other
partial results are known~\cite{DrmotaKauersSpiegelhofer2016,MorgenbesserSpiegelhofer2012, Spiegelhofer2022}.
We also wish to draw attention to the related conjecture by Tu and Deng~\cite{TuDeng2011,TuDeng2012}, coming from cryptography.
This conjecture implies Cusick's conjecture, and holds~\emph{almost surely}~\cite{SpiegelhoferWallner2019}.
Partial results exist~\cite{CusickLiStanica2011,DengYuan2012,FloriThesis,FloriRandriambololonaCohenMesnager2010, ChengHongZhong2015, LiuWu2019},
but the general case is wide open.
Cusick's conjecture arose while T.~W.~Cusick was working on the Tu--Deng conjecture~\cite{CusickLiStanica2011}, and thus the present paper traces back to cryptography.

\subsection{Notation}
For a finite word $\omega$ over $\{\tO,\tL\}$ containing $\tL$, let $\lvert n\rvert_\omega$ denote the number of (overlapping) occurrences of the word $\omega$ in the binary expansion of $n$, padded with suitably many $\tO$s to the left.
Note that in the case $\omega=\tO\tL$, the integer $\lvert n\rvert_\omega$ is the number of maximal blocks of $\tL$s in the binary expansion of $n$,
where a ``block'' is a contiguous finite subsequence.
This is the case as each occurrence of $\tO\tL$ marks the beginning of such a block.

%The condition on the number $N$ in Theorem~\ref{thm_main} therefore just reads $\lvert t\rvert_{\tO\tL}\geq N_0$.

For real $\vartheta$, we will use the notation $\e(\vartheta) = \exp(i \vartheta)$.
Moreover, in this paper, we stick to the convention that $0\in\mathbb N$.
%In order to better understand the binary expansion 

\section{Main result} \label{sec_main}
In the present paper, we are going to establish a central limit-type result in the spirit of Theorem \ref{thm_digitsum_normal}, where the sum-of-digits function $\digitsum$ is replaced by a factor-counting function $\lvert\cdot\rvert_\omega$.
%More precisely, let $\omega$ be a finite word over $\{\tO,\tL\}$ containing $\tL$, and define $a_\omega(n)$ as the number of (possibly overlapping) occurrences of the finite word $\omega$ in the binary expansion of $n$, padded with zeros to the left.
%\BS{Mention Rudin--Shapiro sequence?}
More precisely, we establish a result analogous to Theorem \ref{thm_digitsum_normal}, for $\omega=\tL\tL$.

Let us define
\begin{equation*}
\begin{aligned}
\blocknumber(n)&\eqdef \lvert n\rvert_{\tL\tL}=\#\bigl\{j\geq 0:\delta_{j+1}(n)=\delta_j(n)=\tL\bigr\}.
\end{aligned}
\end{equation*}
This sequence is \texttt{A014081} in Sloane's OEIS\footnote{The Online Encyclopedia of Integer Sequences, \texttt{https://oeis.org}}, and starts with the values
\[
(\blocknumber(n))_{0\leq n<32}=
(0, 0, 0, 1, 0, 0, 1, 2, 0, 0, 0, 1, 1, 1, 2, 3, 0, 0, 0, 1, 0, 0, 1, 2, 1, 1, 1, 2, 2, 2, 3, 4).
\]
For example, $31=(\tL\tL\tL\tL\tL)_2$ has four (overlapping) blocks $\tL\tL$ in binary. Also note that $(\blocknumber(n) \bmod 2)_{n \in \N}$ is the famous Golay--Rudin--Shapiro sequence.

The object of interest will be the difference
\begin{equation}\label{eqn_diff_def}
d(t,n)\eqdef \blocknumber(n+t)-\blocknumber(n),
\end{equation}
As we will show, for each $t\in\N$ and $k \in \Z$ the set
$$  C_t(k) \eqdef \{ n \in \N: d(t,n) = k \} $$
is a finite union of arithmetic progressions (see Proposition \ref{prop_dens_rec} below). %xyz
%\BS{Maybe state this fact as a proposition, also adding the upper bound on the difference of the progressions?}
Consequently, the densities
$$  c_t(k) \eqdef \dens C_t(k) $$
exist and induce a family of probability distributions on $\Z$ with probability mass function $c_t$. In the sequel we will identify these notions and say ``distribution $c_t$'' in short.  

We also define the sequence $(v_t)_{t \in \N}$ by $v_0 = 0$, $v_1 = 3/2$, and
\begin{align} 
    v_{4t}&=v_{2t}, \nonumber\\
    v_{4t+2}&=v_{2t+1}+1, \label{eqn_variance_def} \\
    v_{2t+1}&=\frac{v_t+v_{t+1}}{2}+\frac{3}{4}. \nonumber
\end{align}
As we will see (in Proposition \ref{prop_variance_v} below), $v_t$ is the variance of the associated probability distribution.
%\BS{The ``combined'' variances $v_t$ satisfy a simpler relation than $v_t^{\alpha}, v_t^{\beta}$ so we don't need to define large vectors and matrices here}

\begin{remark}
From the above relations it follows that $(v_t)_{t \in \N}$
%$ $t\mapsto v^{\alpha}_t$ and $t\mapsto v^{\beta}_t$ 
is a $2$-regular sequence~\cite{AlloucheShallit1992}, see~\cite[Theorem~6]{AlloucheShallit2003a}.
\end{remark}

Our main result says that when $\lvert t\rvert_{\tO\tL}$ is large, the distribution $c_t$ is close to a Gaussian distribution with mean $0$ and variance $v_t$.

\begin{theorem}\label{thm_main}
There exist effective absolute constants $C$, $N_0$ such that the following holds.
If the nonnegative integer $t$ satisfies $\lvert t\rvert_{\tO\tL}\geq N_0$, we have
\begin{equation}\label{eqn_main}
\biggl\lvert c_t(k)-\bigl(2\pi v_t\bigr)^{-1/2}\exp\biggl(-\frac{k^2}{2v_t}\biggr)\biggr\rvert
\leq C\frac{(\log N)^2}N,
%c_t(k) = \bigl(2\pi v_t\bigr)^{-1/2}\exp\biggl(-\frac{k^2}{2v_t}\biggr)+\LandauO%\bigl(
%N^{-1}(\log N)^4\bigr),
\end{equation}
where $N=\lvert t\rvert_{\tO\tL}$.
\end{theorem}

\begin{remarks}
\begin{itemize}
\item In analogy to the discussion in~\cite{SpiegelhoferWallner2021} after Theorem 1.2, we see
that the main term is dominant (for large $N$) if
$\lvert k\rvert\leq C_1\sqrt{N\log N}$,
and $C_1$ is any constant in $(0,\sqrt{3}/2)$.
For this, we need both the lower and the upper bound for $v_t$, that is, $3N/4\leq v_t\le5N$, proved in Proposition~\ref{prop_linear_var} further down.
%\BS{I think we need $\lvert k\rvert\leq \delta \frac{\sqrt{3}}{2}\sqrt{N\log N}$ for any fixed $\delta \in (0,1)$ to have main term $\gg N^{-1+\varepsilon}$}
%\LS{Thank you, I mistakenly ignored the term $v_t^{-1/2}$.}
\item 
The statement of the theorem remains true for all $N$ if we choose a larger value for $C$.
Using our method, this necessitates a much larger value, while no mathematical content is gained.
\item In analogy to~\cite[Corollary~1.3]{SpiegelhoferWallner2021},
we obtain the corollary
\[\sum_{k\ge0}c_t(k)=1/2-C_2N^{-1/2}\bigl(\log N\bigr)^5,\]
where $N=\lvert t\rvert_{\tO\tL}$, and $C_2$ is another absolute constant.
\item Is it true that
\begin{equation}\label{eqn_cusick_GRS}
\sum_{k\ge0}c_t(k)>1/2
\end{equation}
for all integers $t\ge0$?
This fundamental question is an analogue of Cusick's conjecture~\eqref{eqn_cusick}
for $r$ in place of $\digitsum$, and forms part of the guiding question~\eqref{eqn_guiding_question}.
Just like Cusick's original conjecture, this question has to remain open for the moment.
By numerical computation,~\eqref{eqn_cusick_GRS} holds for $t<2^{20}$. Among such $t$, the minimal value of the sum is attained for $t = 1013693 = (11110111011110111101)_2$, and equals approximately $0.535$.
%\LS{TODO. Maybe position of the minimum in binary, value}

\item Adapting our proof of Theorem~\ref{thm_main} below to the original situation concerning $\digitsum$, it should be possible to improve the error term in Theorem \ref{thm_digitsum_normal} to $\mathcal O\bigl(N^{-1}(\log N)^2\bigr)$.
\end{itemize}
\end{remarks}

%\BS{Some comment is needed on the main term dominating the error term for small $k$ and large $M$, as in your paper. I think this is where $cM \leq v_t \leq CM$ is used?}
%\BS{Does $\LandauO$ implicitly assume that $M$ is large enough?}
%\BS{Can we say something about $\sum_{k \geq 0} c_t(k)$ as a corollary?}

\section{Proof of the main result}
%In order to state our result, we need means and variances of the
%associated probability distributions.
%probability distribution corresponding to the probability function
%\[k\mapsto\frac{a_t(k)+b_t(k)}2.\]

We first outline the general idea of the proof. Let $\gamma_t$ be the characteristic function of the distribution $c_t$, i.e.,  
$$ \gamma_t(\vartheta) \eqdef \sum_{k\in\mathbb Z}c_t(k)\e(k\vartheta). $$
To approximate $c_t(k)$ we will use the identity
$$ c_t(k) = \frac{1}{2\pi}\int_{-\pi}^{\pi} \gamma_t(\vartheta) \e(-k \vartheta) \, d \vartheta.  $$
We want to show that for $\vartheta$ in a small interval $I = [-\vartheta_0,\vartheta_0]$ around $0$, the function $\gamma_t$ is well approximated by the characteristic function of Gaussian distribution with mean $0$ and variance $v_t$. This is done in Proposition \ref{prop_char_fun_approx}. Evaluating the integral over $I$, where $\gamma_t$ is replaced with said characteristic function, yields (roughly) the main term in \eqref{eqn_main}, while the error term comes from the approximation. On the other hand, the contribution for  $\vartheta \not \in I$ does not exceed said error term due to a strong upper bound on $|\gamma_t(\vartheta)|$, given in Proposition \ref{prop_char_fun_bound}. As discussed in Section \ref{sec_main}, we also establish an upper and lower bound on the variance $v_t$ (given in Proposition \ref{prop_linear_var}) in order to show that the error term in \eqref{eqn_main} is indeed small compared to the main term.

\subsection{Basic properties}
We first show that the functions $c_t$ are indeed well-defined and describe probability distributions, and establish some of their basic properties.
Our starting point is a set of recurrence relations satisfied by the values $d(t,n)$.
%\BS{I commented out the proof that $m^{\alpha}, m^{\alpha}, v^{\alpha}, v^{\beta} $ are the means and variances, since they are not introduced yet}
\begin{lemma}\label{lem_d_rec}
For all $t,n \in \N$, we have $d(0,n) = 0$ and
\begin{equation*}
\begin{array}{r@{\hspace{1mm}}l@{\hspace{1cm}}r@{\hspace{1mm}}l}
d(4t+0,4n+0)&=d(2t+0,2n+0),&
d(4t+2,4n+0)&=d(2t+1,2n+0),\\
d(4t+0,4n+1)&=d(2t+0,2n+0),&
d(4t+2,4n+1)&=d(2t+1,2n+0)+1,\\
d(4t+0,4n+2)&=d(2t+0,2n+1),&
d(4t+2,4n+2)&=d(2t+1,2n+1),\\
d(4t+0,4n+3)&=d(2t+0,2n+1),&
d(4t+2,4n+3)&=d(2t+1,2n+1)-1,\\
&&&\\
d(4t+1,4n+0)&=d(2t+0,2n+0),&
d(4t+3,4n+0)&=d(2t+1,2n+0)+1,\\
d(4t+1,4n+1)&=d(2t+1,2n+0),&
d(4t+3,4n+1)&=d(2t+2,2n+0),\\
d(4t+1,4n+2)&=d(2t+0,2n+1)+1,&
d(4t+3,4n+2)&=d(2t+1,2n+1),\\
d(4t+1,4n+3)&=d(2t+1,2n+1)-1,&
d(4t+3,4n+3)&=d(2t+2,2n+1)-1.\\
\end{array}
\end{equation*}
\end{lemma}
\begin{proof}
All equalities can be quickly derived from $\blocknumber(0) = 0$ and the following relations:
\begin{align*}
\blocknumber(2n) &= \blocknumber(n), \\
\blocknumber(4n+1) &=\blocknumber(n), \\
\blocknumber(4n+3) &= \blocknumber(2n+1)+1.  \qedhere
\end{align*} 
\end{proof}

Note that the relations all involve $d(\cdot,2n)$ or $d(\cdot,2n+1)$ on the right-hand side (though some can be ``merged''). This makes it tricky to directly describe the sets $C_t(k)$ by a collection of recurrence relations, since they have $d(t,n)$ in their definition. Instead, we consider their ``odd'' and ``even'' components: 
\begin{equation*}
   \begin{aligned}
A_t(k)&\eqdef\bigl\{n\in\N:d(t,2n)=k\},\\
B_t(k)&\eqdef\bigl\{n\in\N:d(t,2n+1)=k\},
\end{aligned} 
\end{equation*}
so that $$C_t(k) = 2A_t(k) \cup (2B_t(k)+1).$$
% as
% \begin{equation*}
%    \begin{aligned}
% a_t(k)&\eqdef\dens A_t(k),\\
% b_t(k)&\eqdef\dens B_t(k).
% \end{aligned} 
% \end{equation*}
% We prove that in fact $m^{\alpha}_t$ resp. $m^{\beta}_t$ are the means, and $v^{\alpha}_t$ resp. $v^{\beta}_t$ are the variances of $k\mapsto a_t(k)$ resp. $k\mapsto b_t(k)$.
% \begin{proposition} \label{prop_moments_rec}
% For all $t \in \N$ we have
% \begin{equation}\label{eqn_mean_variance_def}
% \begin{aligned}
% m_t^\alpha&=\sum_{k\in\mathbb Z}ka_t(k),&
% m_t^\beta&=\sum_{k\in\mathbb Z}kb_t(k),\\
% v_t^\alpha&=\sum_{k\in\mathbb Z}\bigl(k-m_t^\alpha\bigr)^2a_t(k),&
% v_t^\beta&=\sum_{k\in\mathbb Z}\bigl(k-m_t^\beta\bigr)^2b_t(k).
% \end{aligned}
% \end{equation}
% \end{proposition}
%The means are constant, while the variances satisfy a matrix recurrence.
%Define
%and
%For integers $0\leq j\leq 3$, $t\geq0$, and $k$, let us define
%\begin{equation}\label{eqn_dens_def}
%\delta_j(k,t):= \dens\bigl\{
%n\in\mathbb N:d(t,4n+j)=k.
%\bigr\}
%\end{equation}
% Let us define
% \begin{align*}
% A_t(k)&:=\bigl\{n\in\N:d(t,2n)=k\},\\
% B_t(k)&:=\bigl\{n\in\N:d(t,2n+1)=k\}.
% \end{align*}
As we will see in Proposition \ref{prop_dens_rec} below, the densities of sets $A_t(k)$ and $B_t(k)$ exist. We denote
%xyz
\begin{align*}
a_t(k)&\eqdef\dens\bigl\{n\in\N:d(t,2n)=k\},\\
b_t(k)&\eqdef\dens\bigl\{n\in\N:d(t,2n+1)=k\},
\end{align*}
which yields
\begin{equation} \label{eqn_density_relation}
   c_t(k) = \frac{a_t(k)+b_t(k)}{2}.
\end{equation}

\begin{proposition} \label{prop_dens_rec}
For all $t \in \N$ and $k \in \Z$ the sets $A_t(k), B_t(k)$ (and thus also $C_t(k)$) are finite unions of arithmetic progressions. Their densities $a_t(k)$ and $b_t(k)$ satisfy the following relations:
%\BS{we can probably put two equations in each line if we move $t$ to the subscript}
\begin{alignat*}{2}
a_{4t}(k)&=\frac12\bigl(a_{2t}(k)+b_{2t}(k)\bigr),  &
b_{4t}(k)&=\frac12\bigl(a_{2t}(k)+b_{2t}(k)\bigr),\\
a_{4t+1}(k)&=\frac12\bigl(a_{2t}(k)+b_{2t}(k-1)\bigr), &
b_{4t+1}(k)&=\frac12\bigl(a_{2t+1}(k)+b_{2t+1}(k+1)\bigr),\\
% a(4t+4,k)&=\frac12\bigl(a(2t+2,k)+b(2t+2,k)\bigr),\\
a_{4t+2}(k)&=\frac12\bigl(a_{2t+1}(k)+b_{2t+1}(k)\bigr), &
b_{4t+2}(k)&=\frac12\bigl(a_{2t+1}(k-1)+b_{2t+1}(k+1)\bigr),\\
a_{4t+3}(k)&=\frac12\bigl(a_{2t+1}(k-1)+b_{2t+1}(k)\bigr), \qquad &
b_{4t+3}(k)&=\frac12\bigl(a_{2t+2}(k)+b_{2t+2}(k+1)\bigr),
% b(4t+4,k)&=\frac12\bigl(a(2t+2,k)+b(2t+2,k)\bigr),
\end{alignat*}
with initial conditions
% \begin{align*}
% a_0(k)=b_0(k) &= 
% \begin{cases}
% 1 &\text{if } k=0, \\
% 0 &\text{if } k\neq 0,
% \end{cases} \\
% a_1(k) &= 
% \begin{cases}
% \frac{1}{2} &\text{if } k=0,1, \\
% 0 &\text{otherwise},
% \end{cases} \\ 
% b_1(k) &= 
% \begin{cases}
% 0 &\text{if } k > 1, \\
% \frac{1}{4} &\text{if } k =1, \\
% 3 \cdot 2^{k-3} &\text{if } k <1.
% \end{cases}
% \end{align*}
$$
a_0(k)=b_0(k) = 
\begin{cases}
1 &\text{if } k=0, \\
0 &\text{if } k\neq 0,
\end{cases}  \qquad
a_1(k) = 
\begin{cases}
\frac{1}{2} &\text{if } k=0,1, \\
0 &\text{otherwise},
\end{cases}  \qquad
b_1(k) = 
\begin{cases}
0 &\text{if } k > 1, \\
\frac{1}{4} &\text{if } k =1, \\
3 \cdot 2^{k-3} &\text{if } k <1.
\end{cases}
$$
\end{proposition}
\begin{proof}
We first deal with the initial conditions. Trivially, we have $A_0(0) = B_0(0) = \N$ and $A_0(k) = B_0(k) = \varnothing$ for $k \neq 0$. It is also easy to check that $A_1(0) = 2\N$, $A_1(1) = 2\N+1$, and $A_1(k) = \varnothing$ for $k\neq 0,1$. Furthermore, we have $B_1(1)=4\N+2$ and $B_1(k)=\varnothing$ for $k > 1$. Finally, for each $k \leq 0$ the set $B_1(k)$ consists of $n \in \N$ such that the binary expansion of $2n+1$ ends with $\tO\tO\tL^{|k|+1}$ or $\tL\tO\tL^{|k|+2}$. Hence, $b_1(k) = 2^{-|k|-2} + 2^{-|k|-3} = 3 \cdot 2^{k-3}$.

To simplify the notation, for $t \in \N$ and $k_A,k_B \in \Z$ let 
$$E_t(k_A,k_B) \eqdef 2 A_{t}(k_A) \cup (2B_t(k_B)+1).$$
Then the identities for $a_t(k)$ and $b_t(k)$ follow straight from corresponding relations for the sets $A_t(k)$ and $B_t(k)$:
\begin{equation*}
\begin{array}{r@{\hspace{1mm}}l@{\hspace{1cm}}r@{\hspace{1mm}}l}
A_{4t}(k) &= E_{2t}(k,k),  &
B_{4t}(k) &= E_{2t}(k,k), \\
A_{4t+1}(k) &= E_{2t}(k,k-1), &
B_{4t+1}(k) &= E_{2t+1}(k,k+1), \\
A_{4t+2}(k) &= E_{2t+1}(k,k),  &
B_{4t+2}(k) &= E_{2t+1}(k-1,k+1), \\
A_{4t+3}(k) &= E_{2t+1}(k-1,k) &
B_{4t+3}(k) &= E_{2t+2}(k,k+1).
\end{array}
\end{equation*}
% \begin{alignat*}{2}
% A_{4t}(k) &= 2 A_{2t}(k) \cup (2B_{2t}(k)+1),  &
% B_{4t}(k) &= 2 A_{2t}(k) \cup (2B_{2t}(k)+1), \\
% A_{4t+1}(k) &= 2 A_{2t}(k) \cup (2B_{2t}(k-1)+1), &
% B_{4t+1}(k) &= 2 A_{2t+1}(k) \cup (2B_{2t+1}(k+1)+1), \\
% A_{4t+2}(k) &= 2 A_{2t+1}(k) \cup (2B_{2t+1}(k)+1),  &
% B_{4t+2}(k) &= 2 A_{2t+1}(k-1) \cup (2B_{2t+1}(k+1)+1), \\
% A_{4t+3}(k) &= 2 A_{2t+1}(k-1) \cup (2B_{2t+1}(k)+1), \qquad &
% B_{4t+3}(k) &= 2 A_{2t+2}(k) \cup (2B_{2t+2}(k+1)+1).
% \end{alignat*}
%one equality per row:
% \begin{align*}
% A_{4t}(k) &= 2 A_{2t}(k) \cup (2B_{2t}(k)+1),  \\
% A_{4t+1}(k) &= 2 A_{2t}(k) \cup (2B_{2t}(k-1)+1), \\
% A_{4t+2}(k) &= 2 A_{2t+1}(k) \cup (2B_{2t+1}(k)+1), \\
% A_{4t+3}(k) &= 2 A_{2t+1}(k-1) \cup (2B_{2t+1}(k)+1), \\
% B_{4t}(k) &= 2 A_{2t}(k) \cup (2B_{2t}(k)+1),  \\
% B_{4t+1}(k) &= 2 A_{2t+1}(k) \cup (2B_{2t+1}(k+1)+1), \\
% B_{4t+2}(k) &= 2 A_{2t+1}(k-1) \cup (2B_{2t+1}(k+1)+1), \\
% B_{4t+3}(k) &= 2 A_{2t+2}(k) \cup (2B_{2t+2}(k)+1).
% \end{alignat*}
Since all these relations are proved similarly, we verify only the one for $B_{4t+2}(k)$ and leave the rest to the reader. We have 
\begin{align*}
B_{4t+2}(k)&= \{ n: d(4t+2,2n+1)=k   \} \\
&=\{ 2n: d(4t+2,4n+1)=k  \} \cup  \{2n+1: d(4t+2,4n+3)=k  \} \\
&= 2\{n: d(2t+1,2n)+1 = k \} \cup ( 2\{n: d(2t+1,2n+1)-1 = k \} +1 )\\
&= 2 A_{2t+1}(k-1) \cup (2B_{2t+1}(k+1)+1) \\
&= E_{2t+1}(k-1,k). \qedhere
\end{align*}
\end{proof}
Note that a bound for the differences of the arithmetic progressions which constitute $A_t(k)$ and $B_t(k)$
can be derived easily from this proof.
These differences are always powers of two, and a rough upper bound is given by $2^{\lvert k\rvert+2\ell(t)+1}$
%$2^{\lvert k\rvert+\ell(t)+2}$,
where $\ell(t)$ is the length of the binary expansion of $t$.

We now define the characteristic functions of the probability distributions $a_t$ and $b_t$:
\begin{align*}
\alpha_t(\vartheta)&:=\sum_{k\in\Z}a_t(k)\e(k\vartheta),\\
\beta_t(\vartheta)&:=\sum_{k\in\Z}b_t(k)\e(k\vartheta).
\end{align*}
Clearly, our function of interest $\gamma_t$ satisfies
$$ \gamma_t(\vartheta) = \frac{\alpha_t(\vartheta)+\beta_t(\vartheta)}{2}. $$

The identities in Proposition \ref{prop_dens_rec} translate to relations for the characteristic functions $\alpha_t$ and $\beta_t$, which we can write concisely using matrix notation. We arrange them into a column vector $S_t(\vartheta)\in\mathbb C^6$, defined by
\[
S_t(\vartheta)
=
\left(
\begin{matrix}
\alpha_{2t+0}(\vartheta) &
\beta_{2t+0}(\vartheta) &
\alpha_{2t+1}(\vartheta) &
\beta_{2t+1}(\vartheta) &
\alpha_{2t+2}(\vartheta) &
\beta_{2t+2}(\vartheta)
\end{matrix}
\right)^T.
\]
We also define $6\times 6$ matrices $D_0(\vartheta), D_1(\vartheta)$ by
$$   
D_0(\vartheta)=\frac12
\left(
\begin{matrix}
1&1&0&0&0&0\\
1&1&0&0&0&0\\
1&\e(\vartheta)&0&0&0&0\\
0&0&1&\e(-\vartheta)&0&0\\
0&0&1&1&0&0\\
0&0&\e(\vartheta)&\e(-\vartheta)&0&0
\end{matrix}
\right), \quad
D_1(\vartheta)=\frac12
\left(
\begin{matrix}
0&0&1&1&0&0\\
0&0&\e(\vartheta)&\e(-\vartheta)&0&0\\
0&0&\e(\vartheta)&1&0&0\\
0&0&0&0&1&\e(-\vartheta)\\
0&0&0&0&1&1\\
0&0&0&0&1&1
\end{matrix}
\right).
$$
We have the following proposition.
\begin{proposition} \label{prop_char_rec}
For all $t \in \N$ we have the recurrence relations
\begin{align*}
S_{2t}(\vartheta) = D_0(\vartheta) S_t(\vartheta), \\
S_{2t+1}(\vartheta) = D_1(\vartheta) S_t(\vartheta),
\end{align*}
with initial conditions
$$S_0(\vartheta)
=
\left(
\begin{matrix}
1 & 1&
\dfrac{\e(\vartheta)+1}{2} &
\dfrac{\e(\vartheta)+1}{2(2-\e(-\vartheta))}&
\dfrac{3\e(\vartheta)+2-\e(-\vartheta)}{4(2-\e(-\vartheta))}&
\dfrac{2\e(2\vartheta)+\e(\vartheta)+\e(-\vartheta)}{4(2-\e(-\vartheta))}
\end{matrix}
\right)^T.
$$
In particular, we have
$$   \alpha_{8t} = \alpha_{4t} = \beta_{8t} = \beta_{4t}. $$
\end{proposition}
%\BS{Do we need to write a proof?}
\begin{proof}
Recurrence relations for $S_t$ as well as initial values $\alpha_0, \beta_0, \alpha_1, \beta_1$ follow immediately from Proposition \ref{prop_dens_rec}. Last two components of $S_0$, namely $\alpha_2, \beta_2$, are obtained by an application of the identity $S_0 = D_0 S_0$ (they only depend on $\alpha_1, \beta_1$). 

Furthermore, we have $\alpha_{4t} = (\alpha_{2t}+\beta_{2t})/2 = \beta_{4t}$ by the relation $S_{2t} = D_0 S_t$. This also implies $\alpha_{8t} = (\alpha_{4t}+\beta_{4t})/2 = \alpha_{4t}$ and similarly $\beta_{8t} =  \beta_{4t}$.
\end{proof}

%Proposition~\ref{prop_char_rec}
%is a straightforward extension of %the corresponding argument for the %sum-of-digits function, and we leave it to the reader.

We move on to give a recursion for the mean and variance of $a_t$ and $b_t$.
We use the notation
$$
m_t^\alpha=\sum_{k\in\mathbb Z}ka_t(k), \qquad
m_t^\beta=\sum_{k\in\mathbb Z}kb_t(k)$$
for the means, and
$$
v_t^\alpha=\sum_{k\in\mathbb Z}\bigl(k-m_t^\alpha\bigr)^2a_t(k), \qquad
v_t^\beta=\sum_{k\in\mathbb Z}\bigl(k-m_t^\beta\bigr)^2b_t(k)
$$
for the variances. As with the characteristic functions, we arrange them in the same way into column vectors
$$
\begin{aligned}
M_t &= \left(
\begin{matrix}
m^\alpha_{2t+0} &
m^\beta_{2t+0} &
m^\alpha_{2t+1} &
m^\beta_{2t+1} &
m^\alpha_{2t+2} &
m^\beta_{2t+2}
\end{matrix}
\right)^T, \\
V_t &= \left(
\begin{matrix}
v^\alpha_{2t+0} &
v^\beta_{2t+0} &
v^\alpha_{2t+1} &
v^\beta_{2t+1} &
v^\alpha_{2t+2} &
v^\beta_{2t+2}
\end{matrix}
\right)^T.
\end{aligned}$$

Using the recursion in Proposition \ref{prop_char_rec}, we can easily obtain relations for $M_t$ and $V_t$. In particular, it turns out that $M_t$ is constant. 

\begin{proposition} \label{prop_moments_rec}
    For all $t \in \N$ we have
    $$  M_t = \left(
\begin{matrix}
0 &
0 &
\dfrac{1}{2} &
-\dfrac{1}{2} &
0 &
0
\end{matrix}
\right)^T $$
and
$$  V_{2t} =  \frac12
\left(
\begin{matrix}
1&1&0&0&0&0\\
1&1&0&0&0&0\\
1&1&0&0&0&0\\
0&0&1&1&0&0\\
0&0&1&1&0&0\\
0&0&1&1&0&0
\end{matrix}
\right) V_t + \frac{1}{4} \left(
\begin{matrix}
0\\
0\\
1\\
4\\
1\\
9
\end{matrix}
\right), \qquad V_{2t+1} =  \frac12
\left(
\begin{matrix}
0&0&1&1&0&0\\
0&0&1&1&0&0\\
0&0&1&1&0&0\\
0&0&0&0&1&1\\
0&0&0&0&1&1\\
0&0&0&0&1&1
\end{matrix}
\right) V_t +  \frac{1}{4} \left(
\begin{matrix}
1\\
9\\
4\\
1\\
0\\
0\end{matrix}
\right), $$
with initial conditions
$$  V_0 = \frac{1}{4}\left(\begin{matrix}
0 &
0 &
1 &
9 &
6 &
14
\end{matrix}
\right)^T.  $$
\end{proposition}
\begin{proof}
We prove the claim for $M_t$ by induction on $t$. The base case $t=0$ is easily verified. Now, by differentiating the first relation in Proposition \ref{prop_char_rec}, for any $t \geq 1$ we have
$$  M_{2t} = -i S_{2t}'(0) = -i D_0'(0) \mathbf{1} + D_0(0) M_t,$$
where $\mathbf{1}$ is the column vector of $1$s of length $6$.
Using the inductive assumption for $M_t$, after a simple calculation we obtain the claimed value of $M_{2t}$. A similar computation also works for $M_{2t+1}$.

% \begin{align*}
% i \cdot M_{2t} &=  S_{2t}'(0) = D_0'(0) S_t(0) + D_0(0) S_t'(0) \\
% &= \frac12
% \left(
% \begin{matrix}
% 0&0&0&0&0&0\\
% 0&0&0&0&0&0\\
% 0&i&0&0&0&0\\
% 0&0&0&-i&0&0\\
% 0&0&0&0&0&0\\
% 0&0&i&-i&0&0
% \end{matrix}
% \right)
% \left(
% \begin{matrix}
% 1 \\
% 1 \\
% 1 \\
% 1 \\
% 1 \\
% 1
% \end{matrix}
% \right)+ \frac12
% \left(
% \begin{matrix}
% 1&1&0&0&0&0\\
% 1&1&0&0&0&0\\
% 1&1&0&0&0&0\\
% 0&0&1&1&0&0\\
% 0&0&1&1&0&0\\
% 0&0&1&1&0&0
% \end{matrix}
% \right)
% \left(
% \begin{matrix}
% 0 \\
% 0 \\
% 1 \\
% -1 \\
% 0 \\
% 0
% \end{matrix}
% \right) = \frac{i}{2}
% \left(
% \begin{matrix}
% 0 \\
% 0 \\
% 1 \\
% -1 \\
% 0 \\
% 0
% \end{matrix}
% \right)
% \end{align*}

Moving on to the variances, for $j=0,1$ we have
$$ S_{2t+j}''(0) = D_j''(0) \mathbf{1}  +2i D_j'(0) M_t + D_j(0) S_t''(0),   $$
After plugging in $S_t''(0) = -V_t-\frac{1}{4}(0,0,1,1,0,0)^T$, and an analogous expression for $S_{2t+j}''(0)$, after a short calculation we get the desired relations.
% In particular,
% \begin{align*}
% v^{\alpha}_{2t} &= \frac{1}{2}(v^{\alpha}_{t}+v^{\beta}_{t}), \\
%  v^{\alpha}_{4t} &=  v^{\beta}_{4t}.
%  \end{align*}
\end{proof}

We can now show that $v_t$, defined by \eqref{eqn_variance_def}, is indeed the variance of the distribution $c_t$.

\begin{proposition} \label{prop_variance_v}
    For all $t$ in $\N$ the distribution $c_t$ has mean $0$ and variance $v_t$. 
\end{proposition}
\begin{proof}
    The mean of $c_t$ is $(m_t^{\alpha} + m_t^{\beta})/2$, which is equal to $0$ by Proposition \ref{prop_moments_rec}.

    Let us momentarily denote the variance by $\widetilde{v}_t$. It satisfies
    $$\widetilde{v}_t = \frac{1}{2}\bigl(v_t^{\alpha}+(m_t^{\alpha})^2+v_t^{\beta}+(m_t^{\beta})^2\bigr) = \frac{v_t^{\alpha}+v_t^{\beta}}{2} + \begin{cases}
       0 &\text{if } t \text{ is even}, \\
       \frac{1}{4} &\text{if } t \text{ is odd}.
   \end{cases}  $$
   Proposition \ref{prop_moments_rec} implies that $\widetilde{v}_0=0=v_0, \widetilde{v}_1=3/2=v_1$, and $\widetilde{v}_t$ satisfies relations \eqref{eqn_variance_def} defining $v_t$, hence we must have $\widetilde{v}_t = v_t$ for all $t \in \N$.
\end{proof}

\subsection{Approximation of the characteristic function}

   The first main ingredient that we need for the central limit-type result is analogous to~\cite[Proposition~3.1]{SpiegelhoferWallner2021}.
We roughly follow the proof of Proposition~2.5 in that paper. We approximate $\gamma_t$ by the characteristic function $\gamma_t^*$ of the Gaussian distribution with the same mean (equal to $0$) and variance $v_t$, namely
$$ \gamma^*_t(\vartheta) = \exp \left(-\frac{v_t}{2} \vartheta^2\right). $$
    % We approximate $\gamma_t$ by the characteristic function $\gamma_t^*$ of the Gaussian distribution with mean $0$ and variance $v_t$, namely
   % $$
   % \gamma^*_t(\vartheta) = \exp \left(-\frac{v_t}{2} \vartheta^2\right) =
   % \left(\alpha^*_t(\vartheta) \beta^*_t(\vartheta)\right)^{1/2} \cdot \begin{cases}
   %     1 &\text{if } t \text{ is even}, \\
   %     \exp(\vartheta^2/4) &\text{if } t \text{ is odd}
   % \end{cases}   
   % $$
   We are interested in bounding the error of approximation
   $$ \widetilde{\gamma}_t(\vartheta) =  \gamma_t(\vartheta) - \gamma^*_t(\vartheta). $$
   %\BS{$=$ or $\in \mathcal{O}$?}
   By definition we have $\widetilde{\gamma}_t(\vartheta) = \mathcal{O}(\theta^3)$ in the sense that its power series expansion only has terms of order $\geq 3$.
   Indeed, $\log \gamma^*_t$ agrees with $\log \gamma_t$, the cumulant generating function, up to terms of order $2$. Hence, after exponentiating both functions still agree up to terms of order $2$. This means that for $|\vartheta| \leq \pi$ we have a bound of the form
   $$ |\widetilde{\gamma}_t (\vartheta)| \leq K_t |\vartheta|^3,$$
   where constant $K_t$ depends on $t$, and we will need to make this dependence more explicit.

In order to do this, we define normal approximations $\alphap_t$ and $\betap_t$ to the characteristic functions
$\alpha_t$ and $\beta_t$,
as well as the errors $\alphaq_t$ and $\betaq_t$ appearing in these approximations.
Let
\begin{align*}
\begin{array}{rl@{\hspace{4em}}rl}
\alphap_t(\vartheta)&\displaystyle\eqdef \exp\biggl(m_t^{\alpha} i \vartheta-\frac12v_t^{\alpha}\vartheta^2\biggr)&
\betap_t(\vartheta)&\displaystyle\eqdef \exp\biggl(m_t^{\beta} i \vartheta-\frac12v_t^{\beta}\vartheta^2\biggr),\\[4mm]
\alphaq_t(\vartheta)&\eqdef\alpha_t(\vartheta)-\alphap_t(\vartheta),&
\betaq_t(\vartheta)&\eqdef\beta_t(\vartheta)-\betap_t(\vartheta)
\end{array}
\end{align*}
(recall that $m_{2t}^{\alpha}=m_{2t}^{\beta}=0$ and $m_{2t+1}^{\alpha}=-m_{2t+1}^{\beta}=1/2$).
Set also
\[
\begin{array}{ccc@{\hspace{1mm}}cccccc@{\hspace{1mm}}ccc}
\Sp_t(\vartheta)&
\eqdef&
\bigl(&
\alphap_{2t+0}(\vartheta)&
\betap_{2t+0}(\vartheta)&
\alphap_{2t+1}(\vartheta)&
\betap_{2t+1}(\vartheta)&
\alphap_{2t+2}(\vartheta)&
\betap_{2t+2}(\vartheta)&
\bigr)^T,\\[2mm]
\Sq_t(\vartheta)&
\eqdef&
\bigl(&
\alphaq_{2t+0}(\vartheta)&
\betaq_{2t+0}(\vartheta)&
\alphaq_{2t+1}(\vartheta)&
\betaq_{2t+1}(\vartheta)&
\alphaq_{2t+2}(\vartheta)&
\betaq_{2t+2}(\vartheta)&
\bigr)^T,
%,\\[2mm]\Xi_t(\vartheta)&\eqdef&\bigl(&&&&&&&\bigr)^T.
\end{array}
\]
so that $\Sq_t(\vartheta) = S_t(\vartheta) - \Sp_t(\vartheta)$.

By Proposition \ref{prop_char_rec} we get the relations
\begin{align*}
    \widetilde{S}_{2t} &= D_0 \widetilde{S}_t - X_{2t}, \\
    \widetilde{S}_{2t+1} &= D_1 \widetilde{S}_t - X_{2t+1},
\end{align*} 
where
\begin{align*}
 X_{2t} &= S^*_{2t} - D_0S^*_{t},   \\
 X_{2t+1} &= S^*_{2t+1} - D_1S^*_{t}.
\end{align*} 
Roughly speaking, $\|X_t(\vartheta)\|_{\infty}$ measures how far the vector of approximations $S^*_{t}(\vartheta)$ is from $S_t(\vartheta)$ after a single application of the recursion. Before we give an upper bound on this quantity, we need an auxiliary lemma.

\begin{lemma} \label{lem_var_diff_bound}
    For all $t \in \N$ we have
    $$ |v^{\alpha}_{t} - v^{\beta}_{t}| \leq 48. $$
\end{lemma}
\begin{proof}
    We first show by induction on $t$ that
    \begin{align*}
        |v^{\alpha}_{t+1} - v^{\alpha}_{t}| &\leq 6,\\
        |v^{\beta}_{t+1} - v^{\beta}_{t}| &\leq 6.
    \end{align*}
    This is easily verified for the base case $t=0$. Let us denote
$$  w_t^{\alpha} = v_{t+1}^\alpha-v_{t}^\alpha, \qquad  w_t^{\beta} = v_{t+1}^\beta - v_{t}^\beta.$$
By Proposition \ref{prop_moments_rec}, for $t \in \N$ we have
\begin{alignat*}{3} 
w_{4t}^{\alpha} &= \frac{1}{4}, &  w_{4t}^{\beta} &= \frac{1}{2} \bigl( w_{2t}^{\alpha}+w_{2t}^{\beta}\bigr)+1, \\
w_{4t+1}^{\alpha} &= \frac{1}{2} \bigl( w_{2t}^{\alpha}+w_{2t}^{\beta}\bigr), &  w_{4t+1}^{\beta} &= \frac{5}{4}, \\
w_{4t+2}^{\alpha} &= \frac{3}{4}, &  w_{4t+2}^{\beta} &= \frac{1}{2} \bigl( w_{2t+1}^{\alpha}+w_{2t+1}^{\beta}\bigr)-2, \\
w_{4t+3}^{\alpha} &= \frac{1}{2} \bigl( w_{2t+1}^{\alpha}+w_{2t+1}^{\beta}\bigr)-1, &\qquad  w_{4t+3}^{\beta} &= -\frac{1}{4}. 
\end{alignat*}

This already implies that $|w_{2t}^{\alpha}| \leq 3/4$ and $|w_{2t+1}^{\beta}| \leq 5/4$. Applying these inequalities combined with the inductive assumption to the remaining identities, we obtain the claim.  For example, we have
$$ |w_{4t+2}^{\beta}| \leq \frac{1}{2} \left(|w_{2t+1}^{\alpha}| + |w_{2t+1}^{\beta}|\right) +2 \leq \frac{1}{2}\left(6+\frac{5}{4}\right)+2 <6.$$

Moving on to the proof of our statement, by Proposition \ref{prop_char_rec} (or Proposition \ref{prop_moments_rec}) we have $v_{4t}^{\alpha} = v _{4t}^{\beta}.$
This implies
$$ |v_{4t+1}^{\alpha} - v_{4t+1}^{\beta}| \leq   |v_{4t+1}^{\alpha} - v_{4t}^{\alpha}| +  |v_{4t+1}^{\beta} - v_{4t}^{\beta}|  \leq 12.  $$
In a similar fashion we can bound $|v_{4t+j}^{\alpha} - v_{4t+j}^{\beta}| $ for $j=2,3$.
\end{proof}

\begin{remark}
    With additional effort it should be possible to prove that $ |v^{\alpha}_{t} - v^{\beta}_{t}| \leq 2$. However, for the purpose of our proof we only need to know that the difference is bounded uniformly in $t$.
\end{remark}

\begin{lemma} \label{lem:single_step}
There exists an absolute constant $C$ such that for all $t \in \N$ and $|\vartheta| \leq \pi$ we have
$$  \|X_t(\vartheta)\|_{\infty} \leq C |\vartheta|^3. $$
\end{lemma}
\begin{proof}
First, observe that each component of $X_t(\vartheta)$, written as a power series, is $\mathcal{O}(\vartheta^3)$ because this is the case for $\widetilde{S}_t,\widetilde{S}_{2t+j}$.

Using Proposition \ref{prop_moments_rec} together with
$$ \log S_t^*(\vartheta)= i M_t \vartheta - \frac{1}{2} V_t \vartheta^2,$$
(where the logarithm is applied component-wise) we obtain the following relations:
\begin{align}
    \log S_{2t}^*(\vartheta) &= D_0(0) \log S_t^*(\vartheta) + \
    \frac{1}{2} \begin{pmatrix}
        0 & 0 & 1 & -1 & 0 & 0
    \end{pmatrix}^T \vartheta -
    \frac{1}{8} \begin{pmatrix}
        0 & 0 & 1 & 4 & 1 & 9
    \end{pmatrix}^T \vartheta^2, \label{eqn_approx_rec} \\
    \log S_{2t+1}^*(\vartheta) &= D_1(0) \log S_t^*(\vartheta) + \
    \frac{i}{2} \begin{pmatrix}
        0 & 0 & 1 & -1 & 0 & 0
    \end{pmatrix}^T \vartheta -
    \frac{1}{8} \begin{pmatrix}
        1 & 9 & 4 & 1 & 0 & 0
    \end{pmatrix}^T \vartheta^2. \nonumber
\end{align}
We now bound individual components of $X_{2t}$ and $X_{2t+1}$. Since the procedure is very similar in each case, we only perform it for only one component. For example let $\xi(\vartheta)$ denote the fourth component of $X_{2t}(\vartheta)$, namely 
$$\xi(\vartheta)=\beta_{4t+1}^*(\vartheta) - \frac{1}{2}( \alpha^*_{2t+1}(\vartheta) + \e(-\vartheta) \beta^*_{2t+1}(\vartheta)).$$ Extracting the fourth component of \eqref{eqn_approx_rec} and exponentiating, we get
$$ \beta_{4t+1}^*(\vartheta) = (\alpha^*_{2t+1}(\vartheta) \beta^*_{2t+1}(\vartheta))^{1/2} \exp\left(-\frac{1}{2}i \vartheta - \frac{1}{2}\vartheta^2\right), $$
where we take the principal value of the square root.
This yields
$$ \xi(\vartheta) = \frac{\alpha^*_{2t+1}(\vartheta)}{2} \left[ 2\left( \frac{\beta^*_{2t+1}(\vartheta)}{\alpha^*_{2t+1}(\vartheta)}\right)^{1/2} \exp\left(-\frac{1}{2}i \vartheta - \frac{1}{2}\vartheta^2\right) -1 - \exp(-i\vartheta) \frac{\beta^*_{2t+1}(\vartheta)}{\alpha^*_{2t+1}(\vartheta)}  \right]. $$
% \begin{align*}
%      &(\alpha^*_{2t+1} \beta^*_{2t+1})^{1/2} \exp(-\frac{1}{2}i \vartheta - \frac{1}{2}\vartheta^2) - \frac{1}{2} \alpha^*_{2t+1} - \frac{1}{2} \e(-\vartheta) \beta^*_{2t+1} = \\
%     &\alpha^*_{2t+1} \left[ \left( \frac{\beta^*_{2t+1}}{\alpha^*_{2t+1}}\right)^{1/2} \exp(-\frac{1}{2}i \vartheta - \frac{1}{2}\vartheta^2) -\frac{1}{2} - \frac{1}{2} \e(-\vartheta) \frac{\beta^*_{2t+1}}{\alpha^*_{2t+1}}  \right].
% \end{align*}  
We have $ |\alpha^*_{2t+1}(\vartheta)| \leq 1$. Also, because $\xi(\theta) = \mathcal{O}(\vartheta^3)$ and $\alpha^*_{2t+1}(\vartheta) = 1 + \mathcal{O}(\vartheta)$, we get
$$ 2\left( \frac{\beta^*_{2t+1}(\vartheta)}{\alpha^*_{2t+1}(\vartheta)}\right)^{1/2} \exp\left(-\frac{1}{2}i \vartheta - \frac{1}{2}\vartheta^2\right) -1 -  \exp(-i\vartheta) \frac{\beta^*_{2t+1}(\vartheta)}{\alpha^*_{2t+1}(\vartheta)} = O(\vartheta^3).  $$
We now consider the terms of order $\geq 3$ of each summand, since the terms of order $\leq 2$ cancel out. First, we have
\begin{align*}    
\left(\frac{\beta^*_{2t+1}}{\alpha^*_{2t+1}}\right)^{1/2} \exp\left(-\frac{1}{2}i \vartheta - \frac{1}{2}\vartheta^2\right) &= \exp \left(-i\vartheta - \frac{1}{4}(v^{\beta}_{2t+1}-v^{\alpha}_{2t+1}+2)\vartheta^2 \right) \\
&= \sum_{k=0}^{\infty} \frac{1}{k!}\left(-i\vartheta - \frac{1}{4}(v^{\beta}_{2t+1}-v^{\alpha}_{2t+1}+2)\vartheta^2 \right)^k.\end{align*}
Because $|\vartheta| \leq \pi$, and $|v^{\beta}_{2t+1}-v^{\alpha}_{2t+1}| \leq 48$ as per Lemma \ref{lem_var_diff_bound}, we get
$$\left|i\vartheta + \frac{1}{4}(v^{\beta}_{2t+1}-v^{\alpha}_{2t+1}+2)\vartheta^2 \right| \leq K |\vartheta|$$
for some absolute constant $K$ (independent of $t$).
As a result, the contribution of terms of order $\geq 3$ are can be bounded by
\begin{align*} &\frac{1}{2}\left|\frac{i}{2}(v^{\beta}_{2t+1}-v^{\alpha}_{2t+1}+2)\vartheta^3 +\left( \frac{1}{4}(v^{\beta}_{2t+1}-v^{\alpha}_{2t+1}+2)\vartheta^2 \right)^2  \right| + \sum_{k=3}^{\infty}\frac{(K|\vartheta|)^k}{k!} \leq \\
& \frac{25}{2} |\vartheta|^3 + \frac{25^2}{8} |\vartheta|^4 + \exp(K \pi) |\vartheta|^3 \leq C_1 |\vartheta|^3
\end{align*}
for a suitable absolute constant $C_1$.

In a similar fashion, we can show that the total contribution of terms of order $\geq 3$ in $\frac{1}{2} \e(-\vartheta) \beta^*_{2t+1}(\vartheta)/\alpha^*_{2t+1}(\vartheta)$ is bounded by $C_2 |\vartheta|^3$ for some absolute constant $C_2$. Therefore,
$$  \left|\beta_{4t+1}^*(\vartheta) - \frac{1}{2} \alpha^*_{2t+1}(\vartheta) - \frac{1}{2} \e(-\vartheta) \beta^*_{2t+1}(\vartheta)\right| \leq (C_1 + C_2) |\vartheta|^3. $$

Repeating this argument for other components of $X_{2t}, X_{2t+1}$ and taking $C$ to be the maximal constant on the right-hand side, we get the result.
   % \begin{align*}
   % \alpha_{4t}^* &= (\alpha_{2t}^*\beta_{2t}^*)^{1/2}, \\
   % \beta_{4t}^* &= (\alpha_{2t}^*\beta_{2t}^*)^{1/2}, \\
   % \alpha_{4t+1}^* &= (\alpha_{2t}^*\beta_{2t}^*)^{1/2} \exp(\frac{1}{2}i\vartheta-\frac{1}{8} \vartheta^2), \\
   % \beta_{4t+1}^* &= (\alpha_{2t+1}^*\beta_{2t+1}^*)^{1/2}\exp(-\frac{1}{2}i\vartheta-\frac{1}{2} \vartheta^2), \\
   % \alpha_{4t+2}^* &= (\alpha_{2t+1}^*\beta_{2t+1}^*)^{1/2}, \\
   % \beta_{4t+2}^* &= (\alpha_{2t+1}^*\beta_{2t+1}^*)^{1/2}, \\
   % \alpha_{4t+3}^* &= (\alpha_{2t+1}^*\beta_{2t+1}^*)^{1/2}\exp(\frac{1}{2}i\vartheta-\frac{1}{2} \vartheta^2), \\
   % \beta_{4t+3}^* &= (\alpha_{2t+2}^*\beta_{2t+2}^*)^{1/2}\exp(-\frac{1}{2}i\vartheta-\frac{1}{8} \vartheta^2). \\
   % \end{align*}   
\end{proof}

We now use the lemma just proved to bound the error of approximation $\widetilde{S}_t(\vartheta)$ after multiple steps of the recursion.

\begin{lemma} \label{lem:multiple_steps}
There exists an absolute constant $K$ such that for all $t \in \N$ and $|\vartheta| \leq \pi$ we have
$$  \|\widetilde{S}_t(\vartheta)\|_{\infty} \leq K N |\vartheta|^3, $$
where $t\in\N$ and $N=\lvert t\rvert_{\tO\tL}$.
\end{lemma}
\begin{proof}
    By simple induction, for any $k \in \N$ we have
    $$\widetilde{S}_{2^k t} = D_0^k \widetilde{S}_t - \sum_{j=1}^{k} D_0^{k-j} X_{2^j t}.$$
    We now show that the sum is bounded uniformly in $k$. First, for $j=1$ and $j=k$ we use Lemma \ref{lem:single_step} , which gives
    \begin{equation} \label{eqn_first_bound}
        \| D_0^{k-j}(\vartheta) X_{2t}(\vartheta)  \|_{\infty} \leq C |\vartheta|^3.
    \end{equation} 
    
    Furthermore, by virtue of Proposition \ref{prop_char_rec} we have $\alpha_{8t}=\alpha_{4t}$ and $\beta_{8t} = \beta_{4t}$, which means that the first two components of $X_{2^j t}$ are $0$ for all $j \geq 2$. Let $\hat{X}_{2^j t}$ denote the vector obtained by deleting these two components. Then we can write in block matrix form
    $$
     D_0^{k-j} X_{2^j t}   = \begin{pmatrix}
         F & 0 \\ G & \hat{D}_0^{k-j} 
     \end{pmatrix} 
     \begin{pmatrix}
         0 \\ \hat{X}_{2^j t} 
     \end{pmatrix} = \begin{pmatrix}
         0 \\ \hat{D}_0^{k-j}\hat{X}_{2^j t} 
     \end{pmatrix},
    $$
    where $F$ is a $2\times 2$ matrix, $G$ a $4\times 2$ matrix, and $\hat{D}_0$ is the submatrix of $D_0$ obtained by deleting its first two rows and columns, namely
    $$\hat{D}_0(\vartheta) = \frac12\begin{pmatrix}
0&0&0&0\\
1&\e(-\vartheta)&0&0\\
1&1&0&0\\
\e(\vartheta)&\e(-\vartheta)&0&0
    \end{pmatrix}.  $$
    Also notice that $\|\hat{D}_0^l(\vartheta)\|_{\infty} = 1/2^{l-1}$ for any $l \geq 1$ and $\vartheta$, which  implies for $j<k$ the inequality
$$  \|  D_0^{k-j}(\vartheta) X_{2^j t}(\vartheta) \|_{\infty} = \|  \hat{D}_0^{k-j}(\vartheta) \hat{X}_{2^j t}(\vartheta) \|_{\infty} \leq \frac{1}{2^{k-j-1}} C |\vartheta|^3.$$
Combining this and \eqref{eqn_first_bound}, we get  
$$  \|\widetilde{S}_{2^k t}\|_{\infty} \leq \|\widetilde{S}_{t}\|_{\infty} + 2C|\vartheta|^3 + \sum_{j=2}^{k-1} \frac{1}{2^{k-j-1}} C |\vartheta|^3 \leq \|\widetilde{S}_{t}\|_{\infty} + 4C|\vartheta|^3.  $$
In other words, appending a block of zeros of arbitrary length to the binary expansion of $t$ increases $\|\widetilde{S}_{t}(\vartheta)\|_{\infty}$ by at most $4C|\vartheta|^3$.

A similar argument also works for appending a block of $1$'s so we omit some of the details. We have the identity
$$\widetilde{S}_{2^k t + 2^{k}-1} = D_1^k \widetilde{S}_t - \sum_{j=1}^{k} D_1^{k-j} X_{2^j t + 2^j -1}.$$
This time, for $j \geq 2$ we have that the last two components of $X_{2^j t + 2^j -1}$ are $0$. Let$\hat{X}_{2^j t + 2^j -1}$ be the vector obtained by deleting these components, and $\hat{D}_1(\vartheta)$ --- the matrix obtained by deleting the last two rows and columns from $D_1(\vartheta)$. Then for we get $2 \leq j \leq k-1$ we get
$$  \|  D_1^{k-j}(\vartheta) X_{2^j t + 2^j-1}(\vartheta) \|_{\infty} = \|  \hat{D}_1^{k-j}(\vartheta) \hat{X}_{2^j t + 2^j-1}(\vartheta) \|_{\infty} \leq \frac{1}{2^{k-j-1}} C |\vartheta|^3.$$
As a consequence, we again arrive at the inequality
$$  \|\widetilde{S}_{2^k t+2^k-1}(\vartheta)\|_{\infty} \leq \|\widetilde{S}_{t}(\vartheta)\|_{\infty} + 4C|\vartheta|^3.  $$
Hence, our claim holds with $K = 4C$, since $\widetilde{S}_0$ is the zero vector.
\end{proof}

Finally, we are ready to give an upper bound on the error $\widetilde{\gamma}_t$ of approximation of $\gamma_t$ by $\gamma^*_t$. We will use the equality
$$  \gamma^*_t(\vartheta) = \left(\alpha^*_t(\vartheta) \beta^*_t(\vartheta)\right)^{1/2} \cdot \begin{cases}
    1 &\text{if } t \text{ is even},\\
    \exp(-\vartheta^2/8) &\text{if } t \text{ is odd},
\end{cases}$$
which follows straight from the definition of $\gamma_t^*$. 

\begin{proposition} \label{prop_char_fun_approx}
There exists an absolute constant $L$ such that for all $t \in \N$ and $|\vartheta| \leq \pi$ we have
$$  |\widetilde{\gamma}_t(\vartheta)| \leq L N |\vartheta|^3, $$
where $t \in \N$ and $N=\lvert t\rvert_{\tO\tL}$.
\end{proposition}
\begin{proof}
    By Lemma \ref{lem:multiple_steps} for all $t \in \N$ we have 
   \begin{align*}
       |\widetilde{\alpha}_t(\vartheta)| &\leq K N |\vartheta|^3, \\
       |\widetilde{\beta}_t(\vartheta)| &\leq K N |\vartheta|^3,
   \end{align*}
   which means that also
   $$ \left| \gamma_t(\vartheta) - \frac{\alpha^*_t(\vartheta)+\beta^*_t(\vartheta)}{2} \right| \leq K N |\vartheta|^3. $$
   
   Furthermore, if $t$ is even, then we get
   $$  \frac{\alpha^*_t(\vartheta)+\beta^*_t(\vartheta)}{2} - \gamma^*_t(\vartheta) = \frac{\alpha^*_t(\vartheta)}{2} \left( \left(\frac{\beta^*_t(\vartheta)}{\alpha^*_t(\vartheta)}\right)^{1/2}-1\right)^2.  $$
   If $t$ is odd, then
   $$  \frac{\alpha^*_t(\vartheta)+\beta^*_t(\vartheta)}{2} - \gamma^*_t(\vartheta) = \frac{\alpha^*_t(\vartheta)}{2} \left(1+ \frac{\beta^*_t(\vartheta)}{\alpha^*_t(\vartheta)}-2\left(\frac{\beta^*_t(\vartheta)}{\alpha^*_t(\vartheta)}\right)^{1/2}  \exp \left(-\frac{\vartheta^2}{8}\right) \right). $$
   In either case, in the same fashion as in Lemma \ref{lem:single_step} we can show that
   $$\left| \frac{\alpha^*_t(\vartheta)+\beta^*_t(\vartheta)}{2} - \gamma^*_t(\vartheta) \right| \leq K_1 |\vartheta|^3$$
   for some constant $K_1$. Choosing $L=K+K_1$, we get the result.
\end{proof}

\subsection{An upper bound on the characteristic function}

We now obtain the second main ingredient of our proof, namely an upper bound on $|\gamma_t(\theta)|$.

\begin{proposition}\label{prop_char_fun_bound}
Assume that $t\in\N$. If $\lvert t\rvert_{\tO\tL}=N$, then for $|\vartheta| \leq \pi$ we have
$$|\gamma_t(\vartheta)| \leq \left(1 - \frac{1}{128} \vartheta^2 \right)^{\lfloor N/2 \rfloor}.
$$
\end{proposition}
\begin{proof}
The statement will follow immediately from the following, more general inequality:
$$ \|S_t(\vartheta)\|_{\infty} \leq \left(1 - \frac{1}{128} \vartheta^2 \right)^{\lfloor N/2 \rfloor} \|S_0(\vartheta)\|_{\infty}. $$
Let $t$ have binary expansion $\varepsilon_{\nu} \varepsilon_{\nu-1} \cdots \varepsilon_1 \varepsilon_0$. Then by Proposition \ref{prop_char_rec} we have
$$ S_t = D_{\varepsilon_0} D_{\varepsilon_1} \cdots D_{\varepsilon_{\nu}}S_0.  $$
Because $D_0 S_0 = S_0$, we can add a leading zero to the expansion of $t$, so that it contains $N$ occurrences of $\tO \tL$. Hence, it contains at least $\lfloor N/2 \rfloor$  non-overlapping strings from the set $\{ \tO\tO\tO\tL,\tO\tL\tO\tL,\tL\tO\tO\tL,\tL\tL\tO\tL   \}$ (strings of length $4$ ending with $\tO \tL$). These in turn correspond to ``disjoint'' subproducts of the form $D_1D_0 D_0 D_0 ,D_1D_0D_1D_0, D_1D_0 D_0D_1,  D_1 D_0D_1D_1  $ in the matrix product. We now bound the row-sum norm of each of these subproducts.

% Consider for example
% $$
% &16D_0 D_0 D_0 D_1 =\\
%  &\left(
% \begin{matrix}
%  0&0 & 4 (\e(\vartheta)+1) & 4(1+\e(-\vartheta)) & 0 & 0 \\
%  0&0 & 4 (\e(\vartheta)+1) & 4(1+\e(-\vartheta)) & 0 &0 \\
% 0& 0 & 2 (\e(2\vartheta)+2\e(\vartheta)+1) & 2 (\e(\vartheta)+2+\e(-\vartheta)) & 0 & 0 \\
% 0 & 0 & \e(2\vartheta)+3\e(\vartheta)+1 + 2 \e(-\vartheta) & \e(\vartheta)+2 + 3 \e(-\vartheta)+\e(-2\vartheta) &\e(-3\vartheta) & \e(-4\vartheta) \\
% 0 & 0 & 2 \e(2\vartheta)+2 \e(\vartheta)+3 &  \e(\vartheta)+ 4 + 2 \e(-\vartheta) & \e(-2\vartheta) & \e(-3\vartheta) \\
% 0 & 0 &  \e(3\vartheta) + 2\e(2\vartheta) + 2\e(\vartheta)+2 \e(-\vartheta) & 
%  \e(2\vartheta) + 2\e(\vartheta) + 1 +2\e(-\vartheta)+\e(-2\vartheta) &  \e(-3\vartheta)& \e(-4\vartheta)  \\
% \end{matrix}
% \right).
% $$

Letting $x = \e(\vartheta)$ for brevity, we have for example
\begin{align*}
&D_1(\vartheta) D_0^3(\vartheta) =\\
 & \frac{1}{16}\left(
\begin{matrix}
 3x+3+x^{-1} &  3x+4 & x^{-2}  & x^{-3} & 0 & 0\\
 2x^2+2x+1+x^{-1} + x^{-2} & 2x^2+2x+1+2x^{-1} & x^{-3} & x^{-4}  & 0 & 0 \\
2x^2+3x+1+x^{-1} &  2x^2+3x+2 & x^{-2} & x^{-3} & 0 & 0 \\
2x+3+x^{-2} & 3x+2+x^{-1}  & x^{-1}+x^{-3} &  x^{-2} + x^{-4} & 0 &0 \\
x^2+2x+2+x^{-1} &  x^2+3x+2  & x^{-1}+x^{-2} & x^{-2} + x^{-3}  & 0 & 0 \\
x^2+2x+2+x^{-1}   &  x^2+3x+2 &  x^{-1}+x^{-2} & x^{-2} + x^{-3} & 0 & 0 \\
\end{matrix}
\right).
\end{align*}
Observe that in each row there is an entry in which contains a subsum of the form $\e(k \vartheta) + \e((k+1) \vartheta)$ for some $k \in \Z$. The absolute value of this expression satisfies
$$|\e(k \vartheta) + \e((k+1) \vartheta)| = |1+\exp(i\vartheta)| = \sqrt{2(1+\cos \vartheta)} = 2 \left| \cos \frac{\vartheta}{2} \right| \leq 2- \frac{\vartheta^2}{8},$$
where we use the inequality $|\cos \varphi| \leq 1-\varphi^2/4$ for $|\varphi|\leq \pi/2$. By trivially bounding the remaining terms in each row, we get
$$ \|D_1(\vartheta) D_0^3(\vartheta)\|_{\infty} \leq \frac{1}{16} \left(16 - \frac{\vartheta^2}{8} \right) = 1 - \frac{\vartheta^2}{128}.$$
The same argument works for the other length-$4$ matrix products. Since $\|D_0(\vartheta)\|_{\infty} = \|D_1(\vartheta)\|_{\infty} = 1$ and $\|S_0(\vartheta)\|_{\infty}=1$, our result follows by submultiplicativity of $\| \cdot \|_{\infty}$.
\end{proof}

\subsection{Bounds on the variance}

Finally, we show that  $v_t \asymp N$, where $N=\lvert t\rvert_{\tO\tL}$ is the number of maximal blocks of $\tL$s in the binary expansion of $t$.

\begin{proposition} \label{prop_linear_var}
    Let $n\in\N$ and $N=\lvert t\rvert_{\tO\tL}$. We have
    $$\frac{3}{4} N \leq v_t \leq 5 N.$$
\end{proposition}
\begin{proof}
We first prove by induction that 
\begin{equation} \label{eqn_var_diff}
    |v_{t+1}-v_t| \leq 3/2.
\end{equation}
This holds for the base case $t=0$. Using the relations \eqref{eqn_variance_def}, we get
    \begin{align*}
        v_{4t+1}-v_{4t} &=  \frac{1}{2}( v_{2t+1}-v_{2t}) + \frac{3}{4}, \\
        v_{4t+2}-v_{4t+1} &=  \frac{1}{2}(v_{2t+1}-v_{2t}) + \frac{1}{4}, \\
        v_{4t+3}-v_{4t+2} &=  \frac{1}{2}(v_{2t+2}-v_{2t+1}) - \frac{1}{4}, \\
        v_{4t+4}-v_{4t+3} &=  \frac{1}{2}(v_{2t+2}-v_{2t+1}) - \frac{3}{4}.
    \end{align*}
Our claim quickly follows from the inductive assumption.

Starting with the lower bound in the statement, by \eqref{eqn_variance_def} we get $v_{2t} \geq v_t$ and 
$$ v_{2t+1} - v_t = \frac{1}{2}(v_{t+1}-v_t) + \frac{3}{4} \geq 0, $$
where we have used \eqref{eqn_var_diff}. In other words, appending a digit to the binary expansion of $t$ does not decrease $v_t$. At the same time, we have
$$ v_{4t+1} = \frac{1}{2}(v_{2t}+v_{2t+1}) + \frac{3}{4} \geq \frac{3}{4}v_t + \frac{1}{4} v_{t+1} + \frac{9}{8}.$$
Subtracting $v_t$ from both sides and using \eqref{eqn_var_diff}, we get $$v_{4t+1}-v_t \geq \frac{3}{4}.$$
Hence, for all $p,q \geq 1$ appending the block $\tO^p \tL^q$ to the binary expansion of $t$  increases $v_t$ by at least $3/4$. The lower bound in the statement follows.

Moving on to the upper bound, for any $k \geq 1$ by \eqref{eqn_variance_def} we have 
$$v_{2^k t}=v_{2t} \in \{v_t, v_t + 1\},$$ 
as well as
$$ |v_{2^k t+2^k-1}-v_t| \leq|v_{2^k(t+1) -1}-v_{2^k(t+1)}| + |v_{2^k(t+1)}- v_{t+1}| + |v_{t+1}-v_t| \leq \frac{3}{2} + 1 + \frac{3}{2} = 4. $$
This means that for all $p,q \geq 1$ appending the block $\tO^p \tL^q$ to the binary expansion of $t$  increases $v_t$ by at most $5$.
\end{proof}

\subsection{Finishing the proof of the main result}
In order to complete the proof of Theorem~\ref{thm_main},
we recall the paper~\cite{SpiegelhoferWallner2021} by the second author and Wallner.
The line of argument we are going to present is analogous, however we establish a refinement of the error bound.
Our Proposition~\ref{prop_char_fun_bound} takes the role of Lemma~2.7 in that paper, while Proposition~\ref{prop_char_fun_approx} is analogous to \cite[Proposition~3.1]{SpiegelhoferWallner2021}.
In our argument, we will see that the Gauss integral
\[
\int_{-\infty}^{+\infty}
\exp\left(-\frac{v_t}{2} \theta^2-ik\vartheta\right)\,\mathrm d\vartheta
\]
is responsible for the emergence of a Gaussian in the main term of~\eqref{eqn_main}.

Let us start with the definition
\[\vartheta_0 = 16 \sqrt{\dfrac{\log N}{N}}\]
of a \emph{cutoff point}, at which we split our integral.
We have
\begin{align*}2\pi c_t(k) &= \int_{-\pi}^{\pi} \gamma_t(\vartheta) \e(-k \vartheta) \, \mathrm d \vartheta \\
&=  \int_{-\vartheta_0}^{\vartheta_0} \gamma_t^*(\vartheta) \e(-k \vartheta) \, \mathrm d \vartheta + \int_{-\vartheta_0}^{\vartheta_0} \widetilde{\gamma}_t(\vartheta) \e(-k \vartheta) \, \mathrm d \vartheta + \int_{\vartheta_0 \leq |\vartheta| \leq \pi} \gamma_t(\vartheta) \e(-k \vartheta) \, \mathrm d \vartheta \\
&= I_1 + I_2 + I_3.
\end{align*}
Expanding the definition of $\gamma_t^*$, we get
$$ I_1 =  \int_{-\infty}^{+\infty} \exp\biggl(-\frac{v_t}{2}\vartheta^2- i k \vartheta\biggr) \, \mathrm d \vartheta - \int_{\lvert\vartheta\rvert \geq \vartheta_0}  \exp\biggl(-\frac{v_t}{2}\vartheta^2- i k \vartheta\biggr) \, \mathrm d \vartheta = I_1^{(1)} - I_1^{(2)}. $$
By completing the square, we have
\[\frac{v_t}{2}\vartheta^2 +  i k \vartheta = \frac{v_t}{2}\left( \vartheta + \frac{ik}{v_t} \right)^2 + \frac{k^2}{2 v_t}.\]
Evaluating a complete Gauss integral, where we may discard the imaginary shift $ik/v_t$, we obtain
\[I_1^{(1)} = \sqrt{\frac{2\pi}{v_t}}\exp\biggl(-\frac{k^2}{2 v_t}\biggr),\]
which gives the main term after division by $2\pi$.
Meanwhile, the first error term satisfies
\[\bigl\lvert I_1^{(2)}\bigr\rvert \leq \int_{\lvert\vartheta\rvert \geq \vartheta_0}  \exp\biggl(-\frac{v_t}{2}\vartheta^2\biggr) \, \mathrm d \vartheta \leq \frac{2}{v_t \theta_0} \exp\biggl(-\frac{v_t}{2} \theta_0^2\biggr) \leq \frac{N^{-96-1/2}}{6\sqrt{\log N}}, \]
where the second inequality follows from the estimate
\[\int_{x_0}^{\infty} \exp\bigl(-cx^2\bigr) \, \mathrm dx \leq  \int_{x_0}^{\infty} \frac{x}{x_0} \exp\bigl(-cx^2\bigr) \, \mathrm dx = \frac{1}{2cx_0} \exp\bigl(-cx_0^2\bigr),\]
valid for any $c, x_0 > 0$, and the third one follows from $v_t \geq \frac{3}{4}N$ and the choice of $\vartheta_0$.

Furthermore, by Proposition \ref{prop_char_fun_approx} we have
\[ |I_2|  \leq \int_{-\vartheta_0}^{\vartheta_0} \bigl\lvert\widetilde{\gamma}_t(\vartheta) \bigr\rvert\, \mathrm d\vartheta \leq \int_{-\vartheta_0}^{\vartheta_0} \bigl\lvert\widetilde{\gamma}_t(\vartheta) \bigr\rvert, \mathrm d\vartheta \leq \int_{-\vartheta_0}^{\vartheta_0} LN\lvert\vartheta\rvert^3 \, \mathrm d\vartheta = \mathcal{O}\bigl(N \vartheta_0^4\bigr)=\mathcal{O}\biggl(\frac{\log^2 N}{N}\biggr).\]

Finally, by Proposition \ref{prop_char_fun_bound} we get
\begin{align*} |I_3| &\leq \int_{\vartheta_0 \leq |\vartheta| \leq \pi} \biggl(1 - \frac{1}{128} \vartheta^2 \biggr)^{\lfloor N/2 \rfloor} \, \mathrm d \vartheta  \leq 2 \int_{\vartheta_0}^{\pi} \exp\biggl(- \frac{N-1}{256} \vartheta^2\biggr) \, \mathrm d \vartheta \leq 2 \pi \exp \biggl(- \frac{N-1}{256} \vartheta_0^2\biggr) \\
&= 2\pi N^{-(N-1)/N} = \mathcal{O}\bigl(N^{-1}\bigr).
\end{align*}
The largest error term is thus $\mathcal{O}\bigl(N^{-1}\log^2 N\bigr)$.
This finishes the proof of our main theorem.

\subsection*{Acknowledgements}
The research topic treated in the present paper was proposed, independently, to the first author (by Maciej Ulas), and to the second author (by Jean-Paul Allouche).

Part of the research for this paper was conducted when B.~Sobolewski was visiting L.~Spiegelhofer at the Montanuniversit\"at Leoben.

%{{{ Bibliography
\bibliographystyle{amsplain}
\bibliography{SobolewskiSpiegelhofer}

\providecommand{\bysame}{\leavevmode\hbox to3em{\hrulefill}\thinspace}
\providecommand{\MR}{\relax\ifhmode\unskip\space\fi MR }
\providecommand{\MRhref}[2]{%
  \href{http://www.ams.org/mathscinet-getitem?mr=#1}{#2}
}
\providecommand{\href}[2]{#2}
\begin{thebibliography}{10}

\bibitem{AlloucheShallit1992}
Jean-Paul Allouche and Jeffrey Shallit, \emph{The ring of {$k$}-regular
  sequences}, Theoret. Comput. Sci. \textbf{98} (1992), no.~2, 163--197.
  \MR{1166363 (94c:11021)}

\bibitem{AlloucheShallit2003a}
\bysame, \emph{The ring of {$k$}-regular sequences. {II}}, Theoret. Comput.
  Sci. \textbf{307} (2003), no.~1, 3--29, Words. \MR{2014728 (2004m:68172)}

\bibitem{Besineau1972}
Jean B\'esineau, \emph{Ind\'ependance statistique d'ensembles li\'es \`a la
  fonction ``somme des chiffres''}, Acta Arith. \textbf{20} (1972), 401--416.
  \MR{0304335}

\bibitem{ChengHongZhong2015}
Kaimin Cheng, Shaofang Hong, and Yuanming Zhong, \emph{A note on the
  {T}u-{D}eng conjecture}, J. Syst. Sci. Complex. \textbf{28} (2015), no.~3,
  702--724. \MR{3341183}

\bibitem{CusickLiStanica2011}
Thomas~W. Cusick, Yuan Li, and Pantelimon St{\u a}nic{\u a}, \emph{On a
  combinatorial conjecture}, Integers \textbf{11} (2011), A17, 17. \MR{2798642}

\bibitem{DengYuan2012}
Guixin Deng and Pingzhi Yuan, \emph{On a combinatorial conjecture of {T}u and
  {D}eng}, Integers \textbf{12} (2012), Paper No. A48, 9. \MR{3083421}

\bibitem{DrmotaKauersSpiegelhofer2016}
Michael Drmota, Manuel Kauers, and Lukas Spiegelhofer, \emph{On a {C}onjecture
  of {C}usick {C}oncerning the {S}um of {D}igits of {$n$} and {$n+t$}}, SIAM J.
  Discrete Math. \textbf{30} (2016), no.~2, 621--649, arXiv:1509.08623.
  \MR{3482392}

\bibitem{DrmotaLarcherPillichshammer2005}
Michael Drmota, Gerhard Larcher, and Friedrich Pillichshammer, \emph{Precise
  distribution properties of the van der {C}orput sequence and related
  sequences}, Manuscripta Math. \textbf{118} (2005), no.~1, 11--41.
  \MR{2171290}

\bibitem{EmmeHubert2018}
Jordan Emme and Pascal Hubert, \emph{Central limit theorem for probability
  measures defined by sum-of-digits function in base 2}, Annali della {S}cuola
  {N}ormale {S}uperiore di {P}isa \textbf{XIX} (2019), no.~2, 757--780.

\bibitem{EmmePrikhodko2017}
Jordan Emme and Alexander Prikhod'ko, \emph{On the {A}symptotic {B}ehavior of
  {D}ensity of {S}ets {D}efined by {S}um-of-digits {F}unction in {B}ase 2},
  Integers \textbf{17} (2017), A58, 28.

\bibitem{FloriThesis}
Jean-Pierre Flori, \emph{Fonctions bool{\'e}ennes, courbes alg{\'e}briques et
  multiplication complexe}, Ph.D. thesis, T{\'e}l{\'e}com ParisTech, 2012.

\bibitem{FloriRandriambololonaCohenMesnager2010}
Jean-Pierre Flori, Hugues Randriam, G\'{e}rard Cohen, and Sihem Mesnager,
  \emph{On a conjecture about binary strings distribution}, Sequences and their
  applications---{SETA} 2010, Lecture Notes in Comput. Sci., vol. 6338,
  Springer, Berlin, 2010, pp.~346--358. \MR{2830750}

\bibitem{Kummer1852}
E.~E. Kummer, \emph{{\"Uber die Erg\"anzungss\"atze zu den allgemeinen
  Reciprocit\"atsgesetzen}}, J. Reine Angew. Math. \textbf{44} (1852), 93--146.

\bibitem{LiuWu2019}
Zhuojun Liu and Baofeng Wu, \emph{Recent results on constructing {B}oolean
  functions with (potentially) optimal algebraic immunity based on
  decompositions of finite fields}, J. Syst. Sci. Complex. \textbf{32} (2019),
  no.~1, 356--374. \MR{3913950}

\bibitem{MorgenbesserSpiegelhofer2012}
Johannes~F. Morgenbesser and Lukas Spiegelhofer, \emph{A reverse order property
  of correlation measures of the sum-of-digits function}, Integers \textbf{12}
  (2012), Paper No. A47, 5. \MR{3083420}

\bibitem{Spiegelhofer2022}
Lukas Spiegelhofer, \emph{A lower bound for {C}usick's conjecture on the digits
  of {$n + t$}}, Math. Proc. Cambridge Philos. Soc. \textbf{172} (2022), no.~1,
  139--161. \MR{4354419}

\bibitem{SpiegelhoferWallner2019}
Lukas Spiegelhofer and Michael Wallner, \emph{The {T}u--{D}eng conjecture holds
  almost surely}, Electron. J. Combin. \textbf{26} (2019), no.~1, Paper 1.28,
  28. \MR{3919615}

\bibitem{SpiegelhoferWallner2021}
\bysame, \emph{The binary digits of $n+t$}, Ann. Sc. Norm. Super. Pisa, Cl.
  Sci. (5) \textbf{24} (2023), no.~1, 1--31.

\bibitem{TuDeng2011}
Ziran Tu and Yingpu Deng, \emph{A conjecture about binary strings and its
  applications on constructing {B}oolean functions with optimal algebraic
  immunity}, Des. Codes Cryptogr. \textbf{60} (2011), no.~1, 1--14.
  \MR{2795745}

\bibitem{TuDeng2012}
\bysame, \emph{Boolean functions optimizing most of the cryptographic
  criteria}, Discrete Appl. Math. \textbf{160} (2012), no.~4-5, 427--435.
  \MR{2876325}

\end{thebibliography}
%}}} Bibliography

\bigskip
%Bartosz Sobolewski
\begin{center}
\begin{tabular}{c}
Jagiellonian University,\\
Krak\'ow, Poland\\
bartosz.sobolewski@uj.edu.pl\\
ORCID iD: \texttt{0000-0002-4911-0062}
\end{tabular}
\end{center}

\smallskip
%Lukas Spiegelhofer
\begin{center}
\begin{tabular}{c}
Department Mathematics and Information Technology,\\
Montanuniversit\"at Leoben,\\
Franz-Josef-Strasse 18, 8700 Leoben, Austria\\
lukas.spiegelhofer@unileoben.ac.at\\
ORCID iD: \texttt{0000-0003-3552-603X}
\end{tabular}
\end{center}

\end{document}